\pdfoutput=1
\documentclass[12pt]{amsart}
\usepackage{amssymb, amsmath, amsthm}
\usepackage{latexsym, amsmath, bm, amssymb, longtable, booktabs,amscd,microtype,booktabs,cases}
\usepackage[square,numbers]{natbib}
\usepackage{graphicx}
\usepackage[dvipsnames]{xcolor}
\usepackage{caption}
\usepackage{dsfont}
\usepackage[all]{xy}
\usepackage{enumerate}
\usepackage{amsthm}
\usepackage{multirow}
\usepackage{tikz}
\usetikzlibrary{matrix,arrows,decorations.pathmorphing}
\usepackage{collectbox}
\usepackage{enumerate}
\usepackage[colorinlistoftodos,prependcaption, textsize=tiny]{todonotes}
\reversemarginpar
\usepackage{amsmath}
\usepackage{enumitem}
\usepackage{hyperref}
\usepackage{longtable}
\theoremstyle{plain}
\newtheorem{theorem}{Theorem}

\newtheorem{corollary}{Corollary}[section]
\newtheorem{proposition}{Proposition}[section]

\theoremstyle{definition}

\newtheorem{definition}{Definition}[section]

\theoremstyle{remark}
\newtheorem{remark}{Remark}[section]

\newcommand{\Z}{\mathbb{Z}}
\newcommand{\la}{\lambda}

\newcommand{\de}{\delta}
\newcommand{\C}{\mathbb{C}}
\newcommand{\dis}{\displaystyle}
\newcommand{\bgop}{\bigoplus}
\newcommand{\ot}{\otimes}
\newcommand{\al}{\alpha}
\newcommand{\mcal}{\mathcal}
\newcommand{\be}{\beta}

\newcommand{\mf}{\mathfrak}
\newcommand{\bs}{\boldsymbol}
\newcommand{\N}{\mathbb{N}}

\newcommand{\op}{\oplus}
\newcommand{\wtil}{\widetilde}
\newcommand{\ov}{\overline}
\usepackage{graphicx}

\begin{document} 
	\title[]{Irreducible modules for map Heisenberg-Virasoro Lie algebras}
	\author[]{ Priyanshu Chakraborty}
	\address{Priyanshu Chakraborty:School of Mathematical Sciences, Ministry of Education Key Laboratory of Mathematics and Engineering Applications and Shanghai Key Laboratory of PMMP,
		East China Normal University, No. 500 Dongchuan Rd., Shanghai 200241, China.}
	
	\email{priyanshu@math.ecnu.edu.cn,        priyanshuc437@.gmail.com}
	
	\maketitle
	\begin{abstract} 
		
		We study irreducible modules for map Heisenberg-Virasoro algebras. In particular, we give a complete classification of irreducible Harish-Chandra modules for map Heisenberg-Virasoro algebras. We will also classify non-weight irreducible modules for map Heisenberg-Virasoro algebras whose restriction on the degree zero part of the universal enveloping algebra of Witt algebra is free of rank 1. 
	\end{abstract}
	{\bf{Notations:}} 
	\begin{itemize}
		\item Throughout this paper we will work with the base field $\C$.
		\item Let $\C, \mathbb{R}, \Z, \Z_{\geq 0}$ denote the set of complex numbers, set of real numbers, set of integers and set of all non-negative integers respectively. 
		\item Let $\N$ denote set of positive integers. For $n \in \N$, $\C^n = \{(x_1, \ldots , x_n) : x_i \in \C, 1 \leq i \leq n\}$ and $\Z_{\geq 0}^{n}$ are defined similarly. 
		\item  Elements of $\C^n $ and $\Z_{\geq 0}^{n}$ are written in boldface.
		\item For any Lie algebra $\mathfrak{g}$, $U(\mathfrak{g})$ will denote universal enveloping algebra of $\mathfrak{g}$.
		
	\end{itemize}
	\section{Introduction} 
	The Lie algebra of polynomial vector fields on $S^1$ is known as Witt algebra. The Witt algebra Der$\C[t^{\pm 1}]$ is an infinite dimensional Lie algebra with basis $d_n=t^{n+1}\frac{d}{dt},$ $\forall n \in \Z$ with the bracket operation $[d_n,d_m]=(m-n)d_{m+n}$. The universal central extension of Witt algebra is known as Virasoro algebra, we denote it by $Vir$. The Virasoro algebra $Vir= $ Der$\C[t^{\pm 1}] \op \C C$ is a Lie algebra with the bracket operation
	$$ [d_n , d_m ] = (m - n)d_{n+m} + \de_{n,-m}\frac{n^3-n}{12}C, \hspace{1cm} [d_n,C]=0, \forall m,n \in \Z. $$
	Harish-Chandra modules for Virasoro Lie algebras are extensively studied, for instance see \cite{MP,OM,CP,MP2,LGZ} and references therein.\\
	The Virasoro Lie algebra and its related Lie algebras have been widely used in several branches of mathematics and physics, for example string theory \cite{MJ}, Conformal field theory \cite{GO}, modular form \cite{KP}, Kac-Moody Lie algebras \cite{VK}, vertex operator algebras and so on.\\
	We concentrate on the Virasoro related Lie algebra called Heisenberg-Virasoro Lie algebra, which is the universal central extension of the Lie algebra $\{f(t)\frac{d}{dt}+g(t):f(t),g(t) \in \C[t^{\pm 1}]\}$ of differential operators of degree at most one, see section 2 for details. Representation of  Heisenberg-Virasoro Lie algebra have been studied in \cite{LZHV,LGZ,CG,CGHW,YB} and references therein. In particular in \cite{LZHV} authors classified irreducible Harish-Chandra modules for Heisenberg-Virasoro Lie algebras and in \cite{CGHW} authors classified irreducible modules for Heisenberg-Virasoro algebras which are free $U(d_0)$-modules of rank 1.\\
	In this paper our aim to study irreducible modules for map Heisenberg-Virasoro algebras. Let $\mf L$ be a Lie algebra and $Z$ be a affine scheme of finite type. Then the Lie algebra of all regular maps from $Z$ to $\mf L$ called the map algebra associated with $Z$ and $\mf L$ and it is isomorphic to the Lie algebra $\mf L \ot B$, where $B=\mcal O_Z$. In recent days study of modules for map algebras gaining interest, for instance see \cite{GLZ,CR,SCPR,RB,CB,RSM}.\\
	In this paper we classify irreducible Harish-Chandra modules for map Heisenberg-Virasoro algebras. In particular we prove that irreducible Harish-Chandra modules are either single point evaluation modules corresponding to Heisenberg-Virasoro modules of intermediate series or a finite tensor product of generalized single point evaluation modules corresponding to highest weight (or lowest weight) modules. To classify irreducible Harish-Chandra modules we used many idea from \cite{1} but arguments are different. On the other hand we classify non-weight irreducibles for map Heisenberg-Virasoro algebras which are free of rank 1 over $U(d_0)$. It should mention here that this kind of free modules were first introduced and studied in \cite{LZ} using fractional representation of $Vir$. Later on this kind of non-weights modules have been studied by several authors and called them $U(h)$-free modules, for instance see \cite{CGHW,CYY,JN,JN2,TZ}. In particular, in \cite{CGHW} authors classified $U(d_0)$ free modules of rank 1 for Heisenberg-Virasoro algebra and in \cite{CYY} authors classified $U(d_0)$ free modules of rank 1 for loop-Virasoro algebra. In this paper we generalize the results of \cite{CGHW}, see section 5.\\
	The  paper is organized as follows. In section 2 we review some important definitions related to map algebras. Furthermore we introduce Verma module for Heisenberg-Virasoro Lie algebra. In section 3 we classify irreducible uniformly bounded modules for map Heisenberg-Virasoro Lie algebras. In section 4 we prove that irreducible Harish-Chandra modules are either highest weight or lowest weight or uniformly bounded modules. Moreover we find a necessary sufficient condition for irreducible highest weight modules to be Harish-Chandra module. In section 5 we classify irreducible non-weight modules for map Heisenberg-Virasoro which are free of rank 1 over $U(d_0)$. Finally we find the isomorphism classes of modules for map Heisenberg-Virasoro algebras which are free of rank 1 over $U(d_0)$.   
	\section{ Preliminaries}
	\begin{definition}
		For a Lie algebra $\mathfrak{g}$ and a commutative associative algebra $B$ we define a Lie algebra $\mf{g} \ot B$ by:
		$$ [g_1\ot b_1,g_2\ot b_2]=[g_1,g_2]\ot b_1b_2 $$
		for all $g_1,g_2 \in \mf g, b_1,b_2 \in B$. We call this algebra as map algebra associated to $\mf g $ and $B$.
		
	\end{definition}
	\begin{definition}
		For a $\mf{g}\ot B$ module $V$, we define 
		\begin{center}
			$Ann_V=\{ f \in B: \mf g \otimes f.V=0 \}  $ and \\
			$Supp_V=$ $\{\bf{m} \in $Spec $B: Ann_V\subseteq \bf{m}\}$.
		\end{center}
		We call $V$ have finite support if $Supp_V$ is finite.
	\end{definition}
	\begin{definition}({\bf Evaluation module})
		Let $B$ be a commutative associative unital finitely generated algebra over $\C$. Let $V$ be a $\mathfrak{g} \otimes B$-module. We call $V$ is a single point evaluation $\mathfrak{g} \otimes B$-module if
		there exists an algebra homomorphism $\eta:B \mapsto \C$ such that $x \otimes b.(v) = \eta(b)(x\otimes1).v$ for all $x \in \mathfrak{g}, b \in B$ and $v \in V$. Equivalently, if $(V, \rho)$ is a $\mathfrak{g} \otimes B$-representation, then $V$ is a single 
		point evaluation module if the Lie algebra homomorphism $\rho: \mathfrak{g} \otimes B \rightarrow \mathrm{End}(V)$ factors through $\mathfrak{g} \otimes B/ \mathfrak{g} \otimes \mathfrak{m} \cong \mathfrak{g} \otimes B/ \mathfrak{m}$, where $\mathfrak{m}$ is a maximal ideal of $B$.\\
		At the same time, we call $V$ is a single point generalized evaluation module if the map $\rho: \mathfrak{g} \otimes B \rightarrow \mathrm{End}(V)$ factors through $\mathfrak{g} \otimes B/ \mathfrak{g} \otimes \mathfrak{m}^k \cong \mathfrak{g} \otimes B/ \mathfrak{m}^k$ for some $k >1$ and $k \in \N$, where $\mathfrak{m}$ is a maximal ideal of $B$.
	\end{definition}
	
	{\bf Map Heisenberg-Virasoro algebra:}\\
	Consider the Lie algebra with basis $\{d_n , I_n, C , C_D , C_I |n \in \Z \}$ and following bracket operations:
	\begin{align}
		[d_n , d_m ] = (m - n)d_{n+m} + \de_{n,-m}\frac{n^3-n}{12}C\\
		[d_n,I_m]=mI_{m+n}+\de_{n,-m}(n^2+n)C_D\\
		[I_n,I_m]=n\de_{n,-m}C_I\\
		[C,d_n]=[C,I_n]=[C_D,d_n]=[C_D,I_n]=[C_I,d_n]=[C_I,I_n]=0,
	\end{align}
	for all $m,n \in \Z$. This Lie algebra is known as Heisenberg-Virasoro Lie algebra. We denote this Lie algebra as $\mathcal{HV}.$  It is easy to observe that $\mathcal{C}=span\{C,C_D,C_I,I_0\}$ is the center of $\mcal {HV}$. Note that the subalgebra spanned by $\{d_n,C : n \in \Z\}$ is the classical Virasoro algebra. We denote Virasoro algebra by $Vir$.\\
	For a commutative associative finitely generated unital algebra $B$, we consider the Lie algebra $\mcal{L_B}=\mcal{HV}\otimes B$. For a $\mcal{L_B}$-module $V$ and $(\la,\la_{I_0},\la_{C},\la_{C_I},\la_{C_D}) \in \C^5$ we define:
	$$V_{(\la,\la_{I_0},\la_{C},\la_{C_I},\la_{C_D})}=\{v\in V:d_0.v=\la v,I_0.v=\la_{I_0}v,C.v= \la_{C}v ,C_I.v=\la_{C_I}v,C_D.v= \la_{C_D}v \}$$ and call $V_{(\la,\la_{I_0},\la_{C},\la_{C_I},\la_{C_D})}$ as weight space of $V$ of weight $(\la,\la_{I_0},\la_{C},\la_{C_I},\la_{C_D})$. We call a module $V$ for $\mcal{L_B}$ as weight module if it is a direct sum of its weight spaces. Note that when $V$ is irreducible Harish-Chandra module, then $\mcal C$ acts as scalar on $V$ and hence we can denote weight spaces just by $V_\la$.\\
	Let us define $$\mcal{HV}^+=span\{d_i,I_i:i >0\}, \hspace{.4cm} \mcal{HV}^0=span\{d_0,I_0,C,C_I,C_D\}, \hspace{.4cm} \mcal{HV}^-=span\{d_i,I_i:i <0\}.$$
	Then $\mcal{HV}= \mcal{HV}^- \op \mcal{HV}^0 \op  \mcal{HV}^+$ gives a triangular decomposition of $\mcal{HV}$. Moreover $\mcal{HV}$ is $\Z$-graded given by the following gradation
	$$\mcal{HV}=\dis{\bigoplus_{i \in \Z}}{\mcal {HV}_i}, \hspace{1cm}  {\mcal{HV}_i}=span\{d_i,I_i \}, i \neq 0,\hspace{1cm} {\mcal {HV}_0}=\mcal{HV}^0.  $$
	Now it is clear that $\mcal{L_B}$ is $\Z$-graded and $\mcal{L_B}=\mcal{L_B}^- \op \mcal{L_B}^0 \op \mcal{L_B}^+,$ where $\mcal{L_B}^{\pm}=\mcal{HV}^{\pm} \ot B$ and $\mcal{L_B}^0=\mcal{HV}^0 \ot B$ gives a triangular decomposition of $\mcal{L_B}$.
	\begin{definition}
		A module $V$ for $\mcal{L_B}$ is said to be a highest weight (respectively lowest weight) module if there exists a non-zero weight vector $v \in V$ such that $\mcal{L_B}^+.v=0$ (respectively $\mcal{L_B}^-.v=0$) and $U(\mcal{L_B}).v=V.$
	\end{definition}
	\begin{definition}({\bf Verma module})
		Let $\phi: \mcal{L_B}^0 \to \mathbb C$ be a one dimensional representation of $\mcal{L_B}^0$. Let $\eta$ be a basis vector of this one dimensional representation.
		Extend $\mathbb C \phi= span\{\eta\}$ to a module for $ \mcal{L_B}^0 \op \mcal{L_B}^+$ by defining action of $\mcal{L_B}^+$ as zero on $\mathbb C\phi$. Then define the Verma module,
		$$M(\phi)=U(\mcal{L_B})\bigotimes_{U(\mcal{L_B}^0 \oplus \mcal{L_B}^+)}\mathbb C \phi.  $$  
		Clearly $M(\phi)$ is a highest weight module with highest weight $(\phi(d_0),\phi(I_0),\phi(C),\phi(C_I),\phi(C_D))$.
		Let $N(\phi)$ be the unique maximal proper sub-module of $M(\phi)$. Then
		$V(\phi)=M(\phi)/N(\phi)  $ is a highest weight irreducible module for $\mcal{L_B}$. Let $v_\phi$ be the highest weight vector of $V(\phi)$.
	\end{definition}
	
	\section{Uniformly bounded modules}
	In this section we classify irreducible uniformly bounded modules for $\mcal{L_B}$. By uniformly bounded modules we mean that dimensions of weight spaces are bounded by some natural number. Before going to uniformly bounded modules for $\mcal{L_B}$ we state the classification result for uniformly bounded modules of $\mcal{HV}$ from \cite{LZHV}.\\
	Let $V(\al,\be,F)$ ($\al,\be,F \in \C$)  be the vector space with basis $\{v_{\al+n}:n \in \Z\}$ and actions of $\mcal{HV}$ on $V(\al,\be,F)$ are defined by:
	\begin{align}
		d_i.v_{\al+k}=(\al+k+\be i)v_{\al+k+i},\\
		I_i.v_{\al+k}=Fv_{\al+k+i},\\
		C.v_{\al+k}=	C_I.v_{\al+k} =	C_D.v_{\al+k}=0.
	\end{align}
	It is well known that $V(\al,\be,F)$ is reducible iff $F=0$, $\al \in \Z$ and $\be =0,1$. Denote the unique (isomorphic)
	non-trivial subquotient module of $V(0,0,0)$ and $V(0,1,0)$ by $V'(0,0,0)$. These modules are known as modules of intermediate series for $\mcal{HV}$.
	\begin{theorem}(\cite{LZHV})
		Let $V$ be a uniformly bounded irreducible module for $\mcal{HV}$. Then $V \simeq V(\al,\be,F)$ or $V'(0,0,0)$. 
	\end{theorem}
	Now we proceed further towards classify irreducible uniformly bounded modules for $\mcal{L_B}$. We start with the following.
	\begin{theorem}\label{t3.1}
		Let $V$ be a non-trivial uniformly bounded irreducible module for $\mcal{L_B}$. Then there exists a cofinite ideal $J$ of $B$ such that $\mcal{HV}\otimes J.V=0.$ In particular, uniformly bounded $\mcal{HV} \otimes B$ modules have finite support.
	\end{theorem}
	\begin{proof}
		Firstly, $V$ is irreducible implies that $V=\dis{\bgop_{m\in \Z}}V_{m}$, where $V_{m}=\{v \in V: d_0.v=(a+m)v  \} $ for some $a \in \C$ and $V_{0}\neq 0$. Consider $V$ as a  $Vir\ot B$-module. Since $V$ is uniformly bounded as $Vir\ot B$-module it has a composition series:
		$$0=V_0\subseteq V_1\subseteq V_2\subseteq.......\subseteq V_k=V. $$
		Then $V_1$ is an irreducible uniformly bounded $Vir\ot B$-submodule of $V$. Hence by Proposition 4.1 of \cite{1}, there exists a cofinite ideal $J$ of $B$ such that $Vir \ot J.V_1=0.$ Now we have $V_1=\dis{\bgop_{m \in \Z}V_{m}\cap{ V_1}}$ with $V_{m}\cap V_1\neq 0$ for some $m \in \mathbb Z$. Assume that $V_{m_0}\cap V_1 \neq 0$.\\
		Consider $R_j=\{b \in B: I_j\ot b.V_1\cap V_{m_0}=0\}$. It is easy to that $R_j$ is an ideal of $B$ for all $j \neq 0$. Moreover it is the kernel of the linear map given by 
		\begin{center}
			
			$B \to Hom(V_1\cap V_{m_0} \, , V_{m_0+j}$)\\
			$b \mapsto (v \mapsto I_j\ot b.v)$.\\
			
		\end{center}Hence $R_j$ is a cofinite ideal of $B$. In particular, dimension of $B/R_j$ is bounded by $N^2$ if $dim V_{m}\leq N$ for all $m \in \Z$.\\
		{\bf Claim: } $\mcal{HV} \ot J^2R_1R_{-1}.V=0$. 
		Note that $[d_j\ot a, I_1\ot r]=I_{j+1}\ot ar$ implies that $JR_1 \subseteq R_{j+1}$ for all $j \geq 0$. Similarly we have $JR_{-1} \subseteq R_{(j-1)}$ for all $j \leq 0$. Hence we have
		\begin{center}
			$J^2R_1R_{-1}\subseteq R_j$ for all $j \neq 0$. 
		\end{center}
		It is easy to see using the bracket operation $$[d_n\ot b,I_m\ot c]=mI_{m+n}\ot bc+\de_{n,-m}(n^2+n)C_D\ot bc,$$ that $I_0 \ot J^2R_1R_{-1}$ and $C_D \ot J^2R_1R_{-1}$ act trivially on $V_1\cap V_{m_0}$. Also $C_I$ is generated by $I_n, I_{-n}$, hence we have $\mcal{HV} \ot J^2R_1R_{-1}.V_1\cap V_{m_0}=0.$ Now note that $W=\{v \in V: \mcal{HV} \ot J^2R_1R_{-1}.v=0 \}$
		is a non-zero $\mcal{L_B}$-submodule of $V$. Hence the claim. This prove the result using the fact that product of finitely many cofinte ideal is cofinite.

	\end{proof}
	\begin{theorem}\label{t3.2}
		Let $V$ be a non-trivial uniformly bounded irreducible module for $\mcal{L_B}$. Then there exists a cofinite ideal $J$ of $B$ supported at a single point such that $\mcal{HV}\otimes J.V=0.$ In particular, non trivial uniformly bounded $\mcal{HV} \otimes B$ modules have support at a single point.
	\end{theorem}
	\begin{proof}
		By Theorem \ref{t3.1}, there exists a cofinite ideal $J$ such that $\mcal{HV} \ot J.V=0$. Since $J$ is cofinite there exists ideals $J_1, \dots, J_k$ supported at distinct points such that $J=J_1\dots J_k$. Now we have $\mcal{HV} \ot B/\mcal{HV} \ot J \simeq \dis{\bigoplus_{i=1}^{k}} \mcal{HV}\ot (B/J_i) $. Therefore to prove this theorem it is sufficient to prove that at most one summand above acts non-trivially on $V$. Let $L_1= \mcal{HV}\ot (B/J_1)$ and $L_2=\dis{\bigoplus_{i=2}^{k}} \mcal{HV}\ot (B/J_i)$. Assume that $L_1$ and $L_2$ both acts non-trivially on $V$. Let $$\de_1=d_0\ot(1+J_1), \hspace{1cm } \de_2=(0,d_0\ot(1+J_2),\dots,d_0\ot(1+J_k)),\hspace{1cm }  \de=\de_1+\de_2, $$

		Then we have $d_0.v=\de.v$ for all $v \in V$. 
		Now consider $V'=\dis{\bigoplus_{i,j \in \Z}}V_{i,j}$, where $V_{i,j}=\{v \in V:\de_1.v=(\alpha+i)v,\de_2.v=(\be+j)v\},$ for some complex numbers $\al, \be$. Note that $span\{\de_1,\de_2\}$ is an abelian subalgebra of $L_1 \op L_2$. Hence there exists a common eigen vector, which implies that $V' \neq 0$.  It is easy to see that $V'$ is a $L_1 \op L_2$ submodule of $V$. Therefore by irreducibility of $V$ we have $V'=V$.\\
		Note that for each $i \in \Z$ we have that $V_{(*,j)}= \dis{\bigoplus_{i \in \Z}}V_{i,j}$ is a $L_1$ submodule of $V$ due to the fact that $[L_1,L_2]=0$. Moreover none of this $V_{(*,j)}$ can be a nonzero trivial module, since if $L_1$ acts trivially on some nonzero vector. Then  irreducibility of $V$ and the fact $[L_1,L_2]=0$ implies that $L_1$ acts trivially on $V$, a contradiction. Now consider $V_{(*,j)}$ as a module for $Vir$ and note that it is a uniformly bounded as $Vir$ module. Therefore by standard results for $Vir$ modules we have $V_{i,j} \neq 0$ for all $\al +i \neq 0$, whenever $V_{(*,j)} \neq 0$. Similarly assuming $L_2$ action on $V$ non-trivial we have $V_{i,j} \neq 0$ for all $\be + j \neq 0$ whenever $V_{(i,*)} \neq 0$. Thus we have $V_{i,j} \neq 0$ when $\al +i \neq 0 \neq \be +j$. Now, $V_{(\al +\be)} \supseteq \dis{\bigoplus_{i \in \Z}}V_{i,-i}$ with right hand side being infinite dimensional, a contradiction. This completes the proof.
	\end{proof}
	Now we make a observation which we will use in the next theorem.
	Let $\mcal{L}$ be a Lie algebra with a vector space decomposition $\mf L=W \op W'$. Then we can pick up order basis of $W$ and $W'$ as $\bs{\Gamma}$ and $\bs{\Gamma'}$ and obtain an order basis of $\mf L$ by defining $b \geq b'$ for all $b \in \bs \Gamma$ and $b'\in \bs \Gamma'$. Let $U_n(W)$ be denote the subspace of $U(\mf L)$ spanned by the monomials of the form $x_1\dots x_s$ such that $x_i \in \bs \Gamma , s \leq n, x_1 \geq \dots \geq x_s$ and set $U(W)=\dis{\cup_{n \geq 0}U_n(W)}$. Similarly define $U(W')$. Then we have $U(\mf L)\simeq U(W)\ot U(W')$, see \cite{1} for more details.

	\begin{proposition}\label{p3.1}
		Let $V$ be an irreducible uniformly bounded Harish-Chandra module for $\mcal{HV} \ot B$ with $B$ finite dimensional. Then $(\mcal{HV} \ot J).V=0$ for any ideal $J$ satisfying $J^2=0.$ 
	\end{proposition}
	\begin{proof}
		Fix some $f \in J$. First note that $I_0 \ot f$ is central and hence acts a scalar. On the other hand $d_0 \ot f$ leaves weight spaces invariant, therefore  there exists a non-zero weight vector $v \in V$ such that $d_0\ot f.v=a_fv$.\\
		{\bf Claim 1: } $d_0 \ot f -a_f$ acts nilpotently on $V$.\\
		It is clear that due to uniformly bounded condition to prove claim 1 it is sufficient to prove that $d_0 \ot f -a_f$ acts locally nilpotently on $V$. Let $B=J+A$ for some vector space compliment $A$ of $J$ in $B$. Then we have a vector space decomposition $\mcal{HV} \ot B =\mcal{HV}' \ot A \op \mathcal C \ot B \op  \mcal{HV} \ot J,$  where $\mcal{HV}'= \dis{\bigoplus_{i \in \Z}\C d_i \bigoplus_{j \in \Z}\C I_j}$. Hence 
		$$U(\mcal{HV} \ot B) =U(\mcal{HV}' \ot A) \ot U(\mathcal C \ot B) \ot  U(\mcal{HV} \ot J).$$
		Since $J^2=0$ and $\mathcal C$ is central, so $\widetilde U=U(\mathcal C \ot B) \ot  U(\mcal{HV} \ot J) $ is an abelian algebra. Since $V$ is irreducible we have $V=U(\mcal{HV} \ot B)v$. We prove that $(d_0 \ot f -a_f)^{n+1}$ acting trivially on $U_n(\mcal{HV}'\ot B)\widetilde U .v$ for all $n \geq 0$. Note that there is nothing to prove for $n=0$. Assume that it is true for $n \leq k$ for some $k \in \N$. Let $u_1,\dots,u_s \in \mcal{HV}'\ot B$, for some $s \leq k$. Then $  (d_0 \ot f -a_f)^{k+1}.u_1\dots u_s.\wtil u.v=(d_0 \ot f -a_f)^{k}[(d_0 \ot f -a_f),u_1\dots u_s.\wtil u].v=(d_0 \ot f -a_f)^{k}[d_0 \ot f ,u_1\dots u_s.\wtil u].v. $ Note that $[d_0 \ot f ,u_1\dots u_s.\wtil u] \in U_{s-1}(\mcal{HV}'\ot B)\widetilde U$, and hence by induction hypothesis $(d_0 \ot f -a_f)^{k}[d_0 \ot f ,u_1\dots u_s.\wtil u].v=0$. This proves claim 1.\\
		{\bf Claim 2:} $a_f=0$ and $C\ot f, C_D \ot f, I_0 \ot f$ acts trivially on $V$.\\
		$a_f=0$ and $C\ot f$ acts trivially on $V$ follows with the same proof of the Proposition 4.5, \cite{1} step 2. Moreover from that proof we have $(d_j \ot f)^r(d_0\ot f-a_f)^{N-r}V=0$ for all $j \in \Z\setminus\{0\}, 0 \leq r \leq N,$ where $N$ is the maximum dimension of weight spaces of $V$. In particular we have $(d_j \ot f)^N.V=0$ for all $j \neq 0$.
		Let $m_j$ be the smallest positive integer such that $(d_j \ot f)^{m_j}.V=0$. Then consider 
		
		$$0=[I_{-j},(d_j \ot f)^{m_j}].V$$
		$$\hspace{2.3cm}=m_j(d_j \ot f)^{m_j-1}[I_{-j},(d_j \ot f)].V$$
		$$\hspace{4.8cm}=m_j(d_j \ot f)^{m_j-1}\{jI_0\ot f -(j^2+j)C_D\ot f\}.V  $$
		$$\hspace{4.8cm}=m_j\{jI_0\ot f -(j^2+j)C_D\ot f\}(d_j \ot f)^{m_j-1}.V , $$
		in this computation we have used the fact that $\mcal{HV} \ot J$ is abelian. Since $m_j$ is the smallest so, there exists a non zero $v $ such that $jI_0 \ot f.v=(j^2+j)C_D\ot f.v$. Since both $I_0 \ot f$ and $C_D\ot f$ acts as scalar the above equality possible for all $j \neq 0$ only when $I_0\ot f .v=C_D\ot f.v=0$.\\
		{\bf Claim 3:} $(\mcal{HV} \ot J)^M.V=0$, where $M=dim B(N-1)+1$.\\
		Since $(d_0 \ot f)^{N}.V=0,$ for $i_1 \neq 0$, we have
		$$0=d_{i_1}.(d_0\ot f)^N.V=-Ni_1(d_{i_1}\ot f)(d_0\ot f)^{N-1}.V.$$ Now consider for $j_1 \neq 0$, 
		$$0=I_{j_1}. (d_{i_1}\ot f)(d_0\ot f)^{N-1}.V= I_{j_1}.(d_0\ot f)^{N-1} (d_{i_1}\ot f).V=j_1(N-1)I_{j_1}\ot f(d_{i_1}\ot f)(d_0\ot f)^{N-2}.V.$$
		Continuing this process we have for $i_1,\dots , i_k \neq 0$ and $j_1,\dots, j_r \neq 0$,
		\begin{align}\label{a3.1}
			(I_{j_1}\ot f)\dots (I_{j_r}\ot f)(d_{i_1}\ot f) \dots (d_{i_k}\ot f).V=0 ,
		\end{align}  where cardinality of the set $\{i_1,\dots, i_k,j_1,\dots, j_r\}=N$. Now by expanding in a basis for $B$ and using equation (\ref{a3.1}) we have 
		$$(I_{j_1}\ot f_1)\dots (I_{j_r}\ot f_r)(d_{i_1}\ot g_1) \dots (d_{i_k}\ot g_k).V=0 ,$$  for all $  i_1, \dots,i_k, j_1, \dots, j_r \in \Z , f_1,\dots,f_k,g_1, \dots,g_r \in B .  $ This proves claim 3.\\
		Thus we have $U(\mcal{HV} \ot B)(\mcal{HV} \ot J)^MU(\mcal{HV} \ot B).V=0$ It is easy to see that $ U(\mcal{HV} \ot B)(\mcal{HV} \ot J)^MU(\mcal{HV} \ot B)=(U(\mcal{HV} \ot B)(\mcal{HV} \ot J)U(\mcal{HV} \ot B))^M$. This implies that $(U(\mcal{HV} \ot B)(\mcal{HV} \ot J)U(\mcal{HV} \ot B))V\neq V$. Since $(U(\mcal{HV} \ot B)(\mcal{HV} \ot J)U(\mcal{HV} \ot B))V$ is a submodule of $V$ and $V$ is irreducible, hence we have $(U(\mcal{HV} \ot B)(\mcal{HV} \ot J)U(\mcal{HV} \ot B)).V=0$. This turns out to the fact that $(\mcal{HV} \ot J).V=0$.

	\end{proof}
	Now we are in a position to describe uniformly bounded irreducible modules for $\mcal{L_B}$. The following theorem states the classification for uniformly bounded modules of $\mcal{L_B}$ and proof follows from verbatim same proof as of Theorem 4.7 of \cite{1} using Theorem \ref{t3.1}, \ref{t3.2} and Proposition \ref{p3.1}.
	\begin{theorem}
		Any uniformly bounded irreducible Harish-Chandra modules for $\mcal{L_B}$ are single point evaluation modules for modules of intermediate series of $\mcal{HV}$.  
	\end{theorem}
	\section{Non-Uniformly bounded modules}
	In this section we classify non-uniformly bounded irreducible Harish-Chandra modules.
	\begin{theorem}
		Any irreducible non-uniformly bounded Harish-Chandra module for $\mcal{L_B}$ is either highest weight module or lowest weight module.
	\end{theorem}
	\begin{proof}
		Let $V$ be an irreducible  non-uniformly Harish-Chandra module for $\mcal{L_B}$. Let $W$ be the minimal Vir-submodule of $V$ such that $V/W$ is trivial $Vir$-module and $T$ be the maximal trivial $Vir$ submodule of $V$. Then by [\cite{MP}, Theorem 3.4], there exists a Vir-module decomposition of $\overline W=W/T \simeq \ov W^- \op \ov W^0 \op \ov W^+,$ where weights of $\ov W^+,\ov W^0, \ov W^-$ are respectively bounded above, uniformly bounded and bounded below. Without loss of generality assume that $\ov W^+ \neq 0$. We denote elements of $\ov W$ as $\ov w$ for $w \in W$.\\
		Since $V$ is irreducible and $C$ is central, hence $C$ acts as a scalar, say $c$. If $c\neq 0$, then $T=0$ and $W=V$. On the other hand if $c=0$ and maximal weight of $W^+ $ is zero with $w \in W$ such that $\ov w$ is the non-zero vector of weight zero. Then $U(Vir)v/U(Vir)v \cap T \subseteq \ov W$ is a non-trivial highest weight module of highest weight zero. Since its simple quotient is trivial, it contain highest weight vector of non-zero highest weight. Thus in any case ($c=0$ or $c \neq 0$) we can choose $w \in W$ such that $\ov w$ is a non-zero highest weight vector of highest weight $\la \neq 0$.\\
		Let $M=U(Vir)w$, then $M/M\cap T$ is a nontrivial highest weight $Vir$-submodule of highest weight $\la$ in $\ov W^+$. Let $M'$ be the maximal $Vir$-submodule of $M$ such that $M'_\la = 0$. Then $M\cap T \subseteq M'$ and $M/M'$ is isomorphic to non-trivial simple $Vir$-module $V(c,\la)$. Since $V(c,\la)$ is not uniformly bounded there exists a $k \in \N$ such that $dim(M/M')_{\la -k} >2dim M_\la$. Then for any $f \in B$ there exists $w_f \in M_{\la -k} \setminus M'_{\la -k}$ such that $ d_k \ot f.w_f=0=I_k \ot f.w_f$. Therefore for $N >>0$ and $j >N$ we have $d_{k+j}.w_f \in M_{\la+j}=0$. Now we get
		$$0=[d_{k+j},d_k\ot f].w_f=-jd_{2k+j}\ot f.w_f   $$
		$$ 0=[d_{k+j},I_k\ot f].w_f=kI_{2k+j}\ot f.w_f. $$
		Since $w_f \in	 M_{\la -k} \setminus M'_{\la -k}$, there exists $z \in \C$ and $i_1, \dots,i_r \in \N$ such that $i_1+\dots+i_r=k$ and $v=zd_{i_1}\dots d_{i_r}.w_f$. Hence we have 
		$$ d_{2k+j}\ot f.v=zd_{2k+j}\ot f d_{i_1}\dots d_{i_r}.w_f=0$$ 
		$$ I_{2k+j}\ot f.v=zI_{2k+j}\ot f d_{i_1}\dots d_{i_r}.w_f=0 ,$$
		the above two expressions are zero due to the fact that when one shift $	d_{2k+j}\ot f$ and $I_{2k+j}\ot f$ from left to right it induces sum of terms with the right most term of each members of the summation are $d_{2k+j+r}.w_f$ and $I_{2k+j+r}.w_f$ for some $r \in \N$ which acts trivially on $w_f$. Thus we have $\mcal{HV}_{i} \ot f.v=0$ for all $i \geq N$, $f \in B$, N is large enough. This completes the proof by Lemma 1.6 of \cite{OM}.
	\end{proof}
	One can observe that highest weight modules $V(\phi)$ may not be Harish-Chandra module for all $\phi$. The following theorem provide a necessary sufficient condition for $V(\phi)$ to be Harish-Chandra module.
	\begin{theorem}\label{tn4.2}
		The irreducible highest weight module $V(\phi)$ is Harish-Chandra module if and only if there exists a cofinite ideal $J$ of $B$ such that $\phi((\mcal{HV})_0 \ot J)=0$.  
	\end{theorem}
	\begin{proof}
		Assume that $V(\phi)$ is a Harish-Chandra module. Let us define $T_1,T_2:B \to V(\phi)_{\phi(d_0)-2} $ by 
		$$T_1(f)=d_{-2}\ot f.v_\phi , \hspace{1cm} T_2(f)=I_{-2}\ot f.v_\phi. $$ Clearly $T_1,T_2$ are linear maps. Let $J=ker T_1\cap kerT_2$. First we assert that $J$ is an ideal of $B$. Since $d_0 \ot B$ leaves the weight spaces invariant and  $V(\phi)_{\phi(d_0)}$ spanned by $v_\phi$, hence for $f \in J$ and $g \in B$ we have
		$$0=[d_0\ot g,d_{-2}\ot f].v_\phi =-2d_{-2}\ot fg.v_\phi $$
		$$ 0=[d_0\ot g,I_{-2}\ot f].v_\phi =-2I_{-2}\ot fg.v_\phi,  $$
		This proves that $J$ is an ideal of $B$. Moreover both $kerT_1$ and $kerT_2$ are cofinite subspaces of $B$ implies that $J$  is a cofinite ideal. Now we prove that $\phi((\mcal{HV})_0 \ot J)=0$. Fix some $f \in J$, now consider following relations:
		\begin{align}
			0=d_2.	d_{-2}\ot f.v_\phi=(-4d_0\ot f+\frac{1}{2}C\ot f).v_\phi\\
			0=d_1^2.	d_{-2}\ot f.v_\phi=6d_0\ot f.v_\phi\\
			0=I_2.	d_{-2}\ot f.v_\phi=(-I_0\ot f-2C_D\ot f).v_\phi\\
			0=d_2.	I_{-2}\ot f.v_\phi=(-2I_0\ot f+6C_D\ot f).v_\phi\\
			0=[I_2,I_{-2}\ot f].v_\phi=C_I\ot f.v_\phi.
		\end{align}
		From the above equations it is clear that $\phi((\mcal{HV})_0 \ot J)=0$.\\ 
		Conversely assume that there exists a cofinite ideal $J$ such that $\phi((\mcal{HV})_0 \ot J)=0$. To complete the proof it is sufficient to prove that $\mcal{HV}\ot J.V(\phi)=0$. Because then $V(\phi)=U((\mcal{HV}^-\ot B/J).v_\phi$ and $J$ is cofinite implies that $V(\phi)$ is Harish-Chandra module. It is clear that $d_n \ot f.v_\phi=0$ and $I_n \ot f.v_\phi=0$ for all $n  >0$ and $f \in J$. Assume that $d_n \ot f.v_\phi=I_n \ot f.v_\phi=0$ for all $n > k, f \in J$, $k \in \Z$. Now for all $b \in B$ we have
		\begin{align}
			d_1 \ot b.(d_{k}\ot f.v_\phi)= (k-1)d_{k+1}\ot bf.v_\phi=0 \\ 
			d_2 \ot b.(d_{k}\ot f.v_\phi)=(k-2)d_{k+2}\ot bf.v_\phi+\frac{1}{2}\delta_{2,-k}C\ot bf.v_\phi=0\\
			I_1 \ot b.(d_{k}\ot f.v_\phi)=-I_{k+1}\ot bf.v_\phi=0.
		\end{align}
		Note that, as a Lie algebra $\mcal{HV}^+$ is generated by $d_1,d_2,I_1$. Therefore $d_k \ot f.v_\phi$ is a highest weight vector and hence $d_k \ot f.v_\phi=0$. Hence by induction we have $d_n \ot J.v_\phi =0$, for all $n \in \Z$. Similarly we have that $I_n \ot J.v_\phi=0$, for all $n \in \Z$. Therefore we have $\mcal{HV}\ot J.v_\phi=0$. Now consider $W=\{v \in V:\mcal{HV}\ot J.v=0\}$, which is a non-zero $\mcal{L_B}$ submodule of $V$. But $V$ is irreducible so $W=V$. This completes the proof.

	\end{proof}
	\begin{corollary}
		Let $B$ be a finite dimensional algebra. Then all highest  weight modules are Harish-Chandra modules.
	\end{corollary}
	The following theorem follows from the similar proof of Theorem 5.3 \cite{1} with the help of Theorem \ref{tn4.2}.
	
	\begin{theorem}
		Any irreducible highest weight Harish-Chandra module is a tensor product of irreducible generalized evaluation highest weight Harish-Chandra modules. 
	\end{theorem}

	After summarizing all the results of section 3 and 4 we have the following theorem.
	
	\begin{theorem}
		Any irreducible Harish-Chandra module for $\mcal{L_B}$ is one of the following:\\
		(1) a single point evaluation module corresponding to a $\mcal{HV}$-module of the intermediate series, or\\
		(2) a finite tensor product of single point generalized evaluation modules corresponding to irreducible highest weight modules (or lowest weight modules).
	\end{theorem}
	\section{Non-weights modules}
	In this section we construct a class of non-weights modules for $\mcal{L_B}$ whose restrictions on $U(d_0)$ are free rank one. For this section we assume that  $B=\C[b_1,b_2,\dots,b_k]$. For scalars $\mu_1, \dots, \mu_k \in \C$ we denote ${\boldsymbol\mu}=(\mu_1,\dots,\mu_k) \in \C^k$. Also for ${\bf r}=(r_1,\dots,r_k) \in \Z^k_{\geq 0}$ we denote $b^{\bf r}=b_1^{r_1}.\dots.b_k^{r_k}$, $\bs \mu^{\bf r}=\mu_1^{r_1}.\dots.\mu_k^{r_k}$ and $|\bf r|$$=\dis{\sum_{i=1}^{k}}r_i$. Before going to the map Virasoro case we first recall from \cite{TZ,LZ,GLZ2} the known result for $Vir$-module which are free of rank 1 over $U(d_0)$. 
	\begin{theorem}\label{t5.1}
		Let $M$ be a $U(Vir )$-module such that restriction of $U(Vir )$ to $U(d_0)$ is of free rank 1. Then $M \simeq \Omega(\la,\al)$, for some $\al \in \C, \la \in \C \setminus \{0\}$. As a vector space $\Omega(\la,\al)=\C[t]$ and actions of elements of $Vir$ on $\Omega(\la,\al)$ are given by:
		\begin{align}
			d_n.f(t)=\la^{n} f(t-n) (t-n\al), \\
			C.f(t)=0.
		\end{align}	
		Furthermore $M$ is irreducible if and only if $ \al \in \C \setminus\{0\}$. If $\al =0$ it has only submodule $t\Omega(\la,0)$.	  
	\end{theorem} \qed \\
	Now we define a $Vir\ot B$-module action on $\C[t]$ by:
	\begin{align}
		d_n \ot b^{\bf r}.f(t)=\bs \mu^{\bf r}\la^{n-|{\bf r}|} f(t-n) (t-n\al), \\
		C\ot b^{\bf r}.f(t)=0
	\end{align}
	for all ${\bf r} \in \Z^k_{\geq 0}, n \in \Z, \bs \mu \in \C^k, \al, \la \in \C, f(t) \in \C[t]$. It is easy to see that this action define a $Vir \ot B$-module structure on $\C[t]$. We denote these modules by $\Omega(\la,\al,\bs \mu)$.

	\begin{theorem}\label{t5.2}
		Let $M$ be a $U(Vir \ot B)$-module such that restriction of $U(Vir \ot B)$ to $U(d_0)$ is free of rank 1. Then $M \simeq \Omega(\la,\al,\bs{\mu})$, for some $\al \in \C, {\bs\mu} \in \C^k, \la \in \C \setminus \{0\}$. Moreover $M$ is irreducible if and only if $ \al\neq 0$.  
	\end{theorem} 
	\begin{proof}
		
		Consider $M$ as a $U(Vir)$ module, then by Theorem \ref{t5.1} we have $M\simeq \Omega(\la,\al)$. Fix some $b \in B$ and let $C\ot b.1=c(t)$. Note that for all $m \geq 1$, $\psi_m:\Omega(\la,\al) \to \Omega(\la,\al)$ defined by $\psi_m(f)=(C\ot b)^m.f$ is a $Vir$ module  homomorphism. Hence Image $\psi_m=c(t)^m\C[t]  $ is a submodule of $\Omega(\la,\al)$. But $\Omega(\la,\al)$ has only submodules $0,\Omega(\la,\al),t\C[t] $. Therefore we have $C\ot b.1$ acts as scalar on $M$ for all $b \in B$. It is easy to observe that 
		\begin{align}
			d_n \ot b.f(t)=	d_n \ot b.f(d_0).1 =f(t-n)d_n \ot b.1, \\
			C\ot b.f(t)=f(t)C\ot b.1,  
		\end{align}
		for all $n \in \Z, b \in B$. Hence to find the actions we need to determine the actions of $d_n \ot b.1$ and $C \ot b.1$ for all $n \in \Z, b \in B$.\\
		Let $d_1\ot b.1=g(t)$ for some fix $b \in B$.  Now consider the relation $[d_1,d_1\ot b].1=0$. This implies that 
		\begin{align}
			(t-\al)g(t-1)=(t-1-\al)g(t).
		\end{align}
		This implies that $(t-\al)|g(t)$, i.e $g(t)=(t-\al)A(t)$, for some $A(t) \in \C[t]$. Putting the value of $g(t)$ in equation (5.7) we have $A(t)=A(t-1)$. Hence we get $A(t)$ is a constant polynomial. Therefore $d_1\ot b.1=A_b(t-\al)$, for some constant $A_b$.  Now consider the relation $[d_{-1},d_1\ot b].1=2d_0\ot b.1$ and use the action of Theorem \ref{t5.1} to get $d_0 \ot b.1 = \la^{-1}A_bt$. From this using the bracket operation $[d_0 \ot b , d_n].1=nd_n\ot b.1$ we have $d_n \ot b.1 = A_b\la^{n-1}(t-n\al)$ for all $n \in \Z$.\\
		{\bf Claim :}  $d_n \ot b^j.1=A_b^j\la^{n-j}(t-n\al)$ for all $j \geq 0, n \in \Z$.\\
		Note that the claim is true for $j=0,1$. Assume that the claim is true for $j \leq s$. Then consider the relation 
		\begin{align*}
			2d_0\ot b^{s+1}.1 &=[d_{-1}\ot b, d_1 \ot b^s ].1 \cr
			&=d_{-1}\ot b.(A_b^s\la^{1-s}(t-\al))-d_1\ot b^s.(A_b\la^{-2}(t+\al))\cr
			&= A_b^{s+1}\la^{-1-s}(t+1-\al)(t+\al)-A_b^{s+1}\la^{-1-s}(t-1+\al)(t-\al) \cr
			&= A_b^{s+1}\la^{-1-s}(2t)
		\end{align*}
		Now using the bracket $[d_0\ot b^{s+1},d_n].1=nd_n\ot b^{s+1}.1$ we have $d_n \ot b^{s+1}=A_b^{s+1}\la^{n-s-1}(t-n\al)$. Hence by induction principal we have the claim.\\
		Therefore there exists scalars $\mu_i$ for $1 \leq i \leq k$ such that $d_n \ot b_i^j=\mu_i^j\la^{n-j}(t-n\al)$ for all $j \geq 0, n \in \Z$.\\
		{\bf Claim :} For $\bf r=$$(r_1,\dots,r_k) \in \Z^k_{\geq 0}$, $d_n \ot b^{\bf r}={\bs \mu}^{\bf r}\la^{n-|\bf r |} (t-n\al)  $ for all $n \in \Z$.\\ Consider the bracket operation $[d_0 \ot b_1^{r_1},d_n \ot b_2^{r_2}].1=nd_n \ot b_1^{r_1}b_2^{r_2}.1$. This gives us $d_n \ot b_1^{r_1}b_2^{r_2}.1= \mu_1^{r_1}\mu_2^{r_2}\la^{n-(r_1+r_2)}(t-n\al)$. Now we consider the bracket operation $[d_0 \ot b_3^{r_3},d_n \ot b_1^{r_1} b_2^{r_2}].1 =nd_n \ot b_1^{r_1}b_2^{r_2}b_3^{r_3}.1 $ and continue this process $n-1$ to get the desired claim.\\
		Moreover considering the bracket operation $[d_{-n},d_n \ot b^{\bf r}].1= 2nd_0\ot b^{\bf r}.1 -\frac{n^3-n}{12}C\ot b^{\bf r}.1$ (for some $n\geq 2$) we get that $C\ot b^{\bf r}.1=0$. This proves that $M\simeq \Omega(\la,\al,\bs \mu)$. Note that when $\al=0$, $t\Omega(\la,0,\bs \mu)$ is a submodule of $\Omega(\la,0,\bs \mu)$. Now by this fact and Theorem \ref{t5.1} we have $M$ is irreducible if and only if $\la,\al\neq 0$. 
		
	\end{proof}
	
	Now we intend towards the non-weight $\mcal{L_B}$ modules whose restrictions on $U(d_0)$ are free of rank 1. We record the following theorem from \cite{CGHW}.
	\begin{theorem}\label{t5.3}
		Let $M$ be a $U\mcal{HV} )$-module such that restriction of $U\mcal{HV} )$ to $U(d_0)$ is of free rank 1. Then $M \simeq \Omega(\la,\al,\be)$, for some $\al, \be \in \C, \la \in \C \setminus \{0\}$. As a vector space $\Omega(\la,\al)=\C[t]$ and actions of elements of $\mcal{HV}$ on $\C[t]$ are given by:
		\begin{align}
			d_n.f(t)=\la^{n} f(t-n) (t-n\al) \\
			I_n.f(t)=\be\la^nf(t-n)\\
			C.f(t)=C_D.f(t)=C_I.f(t)=0.
		\end{align}	
		Furthermore $M$ is irreducible if and only if $ \al \neq 0$ or $\be \neq 0$. If $\al=\be =0$ it has a simple submodule $t\Omega(\la,0,0)$.	  
	\end{theorem} \qed \\
	Now we define a $\mcal{L_B}$-module action on $\C[t]$ by the following actions:
	\begin{align}
		d_n \ot b^{\bf r}.f(t)=\bs \mu^{\bf r}\la^{n-|{\bf r}|} f(t-n) (t-n\al) \\
		I_n \ot b^{\bf r}.f(t)=\bs \mu^{\bf r}\la^{n-|{\bf r}|}\be f(t-n)	\\
		C\ot b^{\bf r}.f(t)=C_D\ot b^{\bf r}.f(t)=C_I\ot b^{\bf r}.f(t)=0,
	\end{align}
	for all ${\bf r} \in \Z^k_{\geq 0}, n \in \Z, \bs \mu \in \C^k,  f(t) \in \C[t]$ and for some $\al,\be  \in \C, \la \in \C \setminus \{0\}$. It is easy to see that this action define a $\mcal{L_B}$-module structure on $\C[t]$. We denote these modules by $\Omega(\la,\al,\bs{\mu},\beta)$.
	\begin{theorem}
		Let $M$ be a $U(\mcal{L_B})$-module such that restriction of $U(\mcal{L_B})$ to $U(d_0)$ is free of rank 1. Then $M \simeq \Omega(\la,\al,\bs{\mu},\beta)$, for some $\al, \be \in \C, {\bs\mu} \in \C^k, \la \in \C \setminus \{0\}$. Moreover $M$ is irreducible if and only if $ \al \neq 0$ or $\be \neq 0$.  If $\al=\be =0$ it has a submodule $t\Omega(\la,0,\bs \mu,0)$
	\end{theorem}
	\begin{proof}
		Consider $M$ as a $U(Vir)$ module, then by Theorem \ref{t5.1} we have $M\simeq \Omega(\la,\al)$. Fix some $b \in B$. Since $\mcal C$ is the center of $\mcal{L_B}$, for all $m \geq 1$, $\psi_m:\Omega(\la,\al) \to \Omega(\la,\al)$ defined by $\psi_m(f)=(\theta\ot b)^m.f$ is a $Vir$-module homomorphism, for $\theta \in \{C_D,I_0,C_I\}$. Now using the fact $\Omega(\la,\al)$ have only submodules $0, \Omega(\la,\al), t\C[t]$, we get $\theta \ot b$ acts as scalar on $M$ for all $b \in B$. Now consider $M$ as a $U(Vir \ot B)$-module which is free of rank 1 over $U(d_0)$. Then by Theorem \ref{t5.2} $M \simeq \Omega((\la,\al,\bs{\mu}).$ Now observe that 
		\begin{align}
			I_n \ot b.f(t)=	I_n \ot b.f(d_0).1 =f(t-n)I_n \ot b.1, \\
			C_D\ot b.f(t)=f(t)C_D\ot b.1\\
			C_I\ot b.f(t)=f(t)C_I\ot b.1,
		\end{align}
		for all $f(t) \in \C[t], b \in B, n \in \Z$. \\
		{\bf Case I:} Let $I_n.1=0$, for some $n \neq 0$ and $n \in \Z$. Then we consider the bracket operation $[d_i \ot b^{\bf r},I_n].1=nI_{n+i} \ot b^{\bf r}.1	$, which implies that $I_{n+i} \ot b^{\bf r}.1=0$, for all $i \neq -n$. Now it is easy to see using  the brackets of $\mcal{L_B}$ that $C_I \ot b^{\bf r}, C_D\ot b^{\bf r}, I_0 \ot b^{\bf r}$ acts trivially on $M$. Hence in this case $M \simeq \Omega(\la,\al, \bs \mu,0)$.\\
		{\bf Case II:} Let $I_n.1 =g(t) \neq 0$, for $n \neq 0$. Let $g(t)=\dis{\sum_{i=0}^{r}}a_it^i$ be a polynomial of degree $r$. Now consider the following:
		\begin{align*}
			(nI_0+(n^2-n)C_D).1 &=	[d_{-n},I_n].1  \cr
			&= \la^{-n}\dis{\sum_{i=0}^{r}}a_i(t+n)^i(t+n\al)-\la^{-n}(t-n+n\al)\dis{\sum_{i=0}^{r}}a_it^i \cr
			&=\la^{-n}\dis{\sum_{i=0}^{r}}a_i(t+n)^{i+1}+(n\al -n)\la^{-n}\dis{\sum_{i=0}^{r}}a_i(t+n)^i \cr
			&-\la^{-n}\dis{\sum_{i=0}^{r}}a_it^{i+1} -(n\al-n)\la^{-n}\dis{\sum_{i=0}^{r}}a_it^i
		\end{align*}
		Since LHS of above equation is constant, differentiation with respect to $t$ yields the value zero. Now differentiating RHS $r$-times we have 
		\begin{align*}
			0&=	(r+1)!a_r(t+n)+a_{r-1}r!+(n\al-n)a_rr!-a_r(r+1)!t-a_{r-1}r!-(n\al-n)a_rr!\cr
			&=na_r(r+1)!,
		\end{align*}
		here $r!$ represents the factorial of a non-negative integer.
		This implies that $g(t)$ is a polynomial of lower degree. Continuing this process we get that $g(t)$ is constant.\\
		Let $I_n.1=\beta_n$ for all $ n\in \Z, \be_n \in \C$. Now consider $[d_m,I_n].1=nI_{m+n}.1+(m^2+m)C_D.1$. This implies that  
		\begin{align}\label{a5.11}
			n\be_{n}\la^m= n\be_{m+n}+(m^2+m)C_D.1.
		\end{align} 
		In particular taking $m=-1$ we have, $\be_n=\la\be_{n-1}$ for all $n \neq 0$. Then from the equation (\ref{a5.11}) we have $C_D.1=0$ Therefore we have a recurrence relation $	n\be_{n}\la^m= n\be_{m+n}$, this implies that $\be_n =\la^n\be$ where $\be =\be_0$.\\
		Now consider the action $[d_0 \ot b^{\bf r},I_n].1=nI_n \ot b^{\bf r}.1$. From this we obtain that $I_n\ot b^{\bf r}=\bs \mu^{\bf r}\la^{n-|{\bf r}|}\be$, for all $n \neq 0$. Now consider  $[d_{-1},I_1 \ot b^{\bf r}].1=I_0\ot b^{\bf r}.1$ to find $I_0\ot b^{\bf r}.1= \bs \mu^{\bf r}\la^{-|\bf {r}|}\beta$. Now it is easy to see that $C_D \ot b^{\bf r}$ and $C_I \ot b^{\bf r} $ acts trivially on $M$. Hence $M\simeq \Omega(\la, \al, \bs \mu, \be)$. Rest part of the theorem follows from Theorem \ref{t5.3}.

	\end{proof}
	\begin{theorem}\label{t5.5}
		$\Omega(\la,\al,\bs\mu,\be) \simeq \Omega(\la',\al',\bs\mu',\be')$ if and only if $\la=\la',\al=\al',\bs\mu=\bs\mu',\be=\be'.$
	\end{theorem}
	\begin{proof}
		Let $\phi: 	\Omega(\la,\al,\bs\mu,\be) \to \Omega(\la',\al',\bs\mu',\be')$ be the isomorphism with its inverse $\phi^{-1}$. Consider the restriction of $\phi$ to $Vir$-module, then from \cite{LZ} we have $\la=\la',\al=\al'.$ Let $\bs \mu=(\mu_1,\dots,\mu_k)$ and $\bs \mu'=(\mu_1',\dots,\mu_k')$. Note that the following relations holds:
		\begin{align}
			\phi(f(t))=\phi(f(d_0)1)=f(t)\phi(1) \\
			\phi^{-1}(f(t))=\phi^{-1}(f(d_0)1)=f(t)\phi^{-1}(1).
		\end{align}
		In particular we have $\phi^{-1}(\phi(1))=\phi(1)\phi^{-1}(1)=1$, which implies that $\phi(1) $ is a non-zero scalar. Now consider
		\begin{align*}
			&\phi(d_1 \ot b_i.1)=d_1\ot b_i.\phi(1) \cr
			\implies &\mu_i(t-\al)\phi(1)=\mu_i'(t-\al')\phi(1)\cr
			\implies & \mu_i=\mu_i', 
		\end{align*}
		for all $1 \leq i \leq k$, using the fact that $\al=\al'$. Again consider the equation $\phi(I_n.1)=I_n.\phi(1)$ and deduce that $\be =\be'$. This completes the proof.
		
	\end{proof}
	\begin{remark}
			
			It is easy to observe from the proof of Theorem \ref{t5.5} that $Vir \ot B$ modules $\Omega(\la,\al,\bs\mu) \simeq \Omega(\la',\al',\bs\mu')$ if and only if $\la=\la',\al=\al',\bs\mu=\bs\mu'.$
			
		\end{remark}
		
		\vspace{2cm}
		{\bf Acknowledgments:} The author would like to thank Dr. Sachin S. Sharma for his manuscript regarding proof of uniformly bounded modules for $\mcal{L_B}$ modulo its center.

	\end{document}